\documentclass{amsart}
\usepackage{amssymb}
\usepackage{tikz}

\newcommand*{\DashedArrow}[1][]{\mathbin{\tikz [baseline=-0.25ex,-latex, dashed,#1] \draw [#1] (0pt,
0.5ex) -- (1.3em,0.5ex);}}
\newcommand{\cqd}{\qed}
\usepackage{amscd}
\usepackage{empheq,amsmath}
\usepackage{graphicx}
\usepackage[cmtip,all]{xy}

\newtheorem{theorem}{Theorem}[section]
\newtheorem{lemma}[theorem]{Lemma}

\theoremstyle{definition}

\newtheorem{definition}[theorem]{Definition}

\theoremstyle{proposition}
\newtheorem{proposition}[theorem]{Proposition}

\newtheorem{remark}[theorem]{Remark}
\newtheorem{cor}[theorem]{Corollary}

\numberwithin{equation}{section}

\theoremstyle{main}
\newtheorem{main}{Theorem}

\begin{document}

\title[Branched pull-back components]{Branched pull-back components of the space of codimension 1 foliations on $\mathbb P^n$}

\author[W. Costa e Silva]{W. Costa e Silva}
\address{IRMAR, Universit\'e Rennes 1, Campus de Beaulieu, 35042 Rennes Cedex France}
\curraddr{IRMAR, Universit\'e Rennes 1, Campus de Beaulieu, 35042 Rennes Cedex France}
\email{wancossil@gmail.com}
\thanks{I am deeply grateful to A. Lins Neto and D. Cerveau for the discussions, suggestions and comments. This work was developed at IMPA (Rio de Janeiro, Brazil) and was supported by IMPA and CNPq (process number 142250/2005-8)}

\date{}

\dedicatory{To my mother}

\begin{abstract} Let $\mathcal{F}$ be written as $ f^{*}\mathcal{G}$, where $\mathcal{G}$ is a foliation in $ {\mathbb P^2}$ with three invariant lines in general position, say $(XYZ)=0$, and $f:{\mathbb P^n}\DashedArrow[->,densely dashed]{\mathbb P^2}$, $f=(F^\alpha_{0}:F^\beta_{1}:F^\gamma_{2})$ is a nonlinear rational map. Using local stability results of singular holomorphic foliations, we prove that: if $n\geq 3$, the foliation $\mathcal{F}$ is globally stable under holomorphic deformations.  As a consequence we obtain new irreducible componentes for the space of codimension one foliations on  $\mathbb P^n$. We present also a result which characterizes holomorphic foliations on ${\mathbb P^n}, n\geq 3$ which can be obtained as a pull back of foliations on $ {\mathbb P^2}$ of degree $d\geq2$ with three invariant lines in general position. \end{abstract}

\maketitle

\tableofcontents

\section{Introduction}

Let $\mathcal F$ be a holomorphic singular foliation on $\mathbb P^{n}$ of codimension $1$, $\Pi_{n}:\mathbb C^{n+1}\backslash \left\{0\right\} \to \mathbb P^{n}$  be the natural projection and $\mathcal F^*=\Pi_{n}^*\left(\mathcal F\right).$ It is known that $\mathcal F^*$ can be defined by an integrable $1-$form $\Omega=\sum_{j=0}^{n}A_{j}dz_{j}$ where the $A_{j}'s$ are homogeneous polynomials of the same degree satisfying the Euler condition:
\begin{equation}\label{rel-1}
\sum_{j=0}^{n}z_{j}A_{j}\equiv0.
\end{equation}
The singular set $S(\mathcal F)$ is given by $S(\mathcal F)=\left\{A_{0}=...=A_{n}=0\right\}$
and is such that codim$\left(S\left(\mathcal F\right)\right)\geq2$. The integrability condition is given by 
\begin{equation}\label{rel-2}
\Omega \wedge d\Omega=0.
\end{equation}

The form $\Omega$ will be called a homogeneous expression of $\mathcal F.$
The degree of $\mathcal F$ is, by definition, the number of tangencies (counted with multiplicities) of a generic linearly embedded $\mathbb P^1$ with $\mathcal F.$ If we denote it by $deg(\mathcal F)$ then $deg\left(\mathcal F\right)=d-1$, where $d=deg\left(A_{0}\right)=...=deg\left(A_{n}\right).$ We denote the space of foliations of a fixed degree $k$ in  $\mathbb P^n$ by $\mathbb{F}{\rm{ol}}\left(k,n\right).$ Due to the integrability condition and the fact that $S\left(\mathcal F\right)$ has codimension $\geq2$, we see that $\mathbb{F}{\rm{ol}}\left(k,n\right)$ can be identified with a Zariski's open set in the variety obtained by projectivizing  the space of forms $\Omega$ which satisfy (\ref{rel-1}) and (\ref{rel-2}). It is in fact an intersection of quadrics. To obtain a satisfactory description of $\mathbb{F}{\rm{ol}}\left(d;n\right)$ (for example, to talk about deformations) it would be reasonable to know the decomposition of $\mathbb{F}{\rm{ol}}\left(d;n\right)$ in irreducible components. This leads us to the following:
\vskip0.2cm
\noindent \textbf{{\underline{Problem}: \emph{Describe and classify the irreducible components of\break $\mathbb{F}{\rm{ol}}\left(k;n\right)$ $k\geq 3$ on ${\mathbb P^n}$, $n\geq 3.$}}}
\vskip0.2cm
 In the  paper \cite{cln}, the authors proved that the space of holomorphic codimension one foliations of degree $2$ on  ${\mathbb P^n}, n\geq 3$, has six irreducible components, which can be described by geometric and dynamic properties of a generic element. We refer the curious reader to \cite{cln} and \cite{ln}  for a detailed description of them.
 There are known families of irreducible components in which the typical element is a pull-back of a foliation on ${\mathbb P^2}$ by a rational map. Given a generic rational map $f: {\mathbb P^n}  \DashedArrow[->,densely dashed    ]   {\mathbb P^2}$ of degree $\nu\geq1$, it can be written in homogeneous coordinates as $f=(F_0,F_1,F_2)$ where $F_0,F_1$ and $F_2$ are homogeneous polynomials of degree $\nu$. Now consider a foliation $\mathcal G$ on ${\mathbb P^2}$ of degree $d\geq2.$ We can associate to the pair $(f,\mathcal G)$ the pull-back foliation $\mathcal F=f^{\ast}\mathcal G$.
The degree of the foliation $\mathcal F$ is $\nu(d+2)-2$ as proved in \cite{clne}.
Denote by $PB(d,\nu;n)$ the closure in $\mathbb{F}{\rm{ol}}\left(\nu(d+2)-2,n\right)$, $n \geq 3$ of the set of foliations $\mathcal F$ of the form $f^{\ast}\mathcal G$. Since $(f,\mathcal G)\to f^{\ast}\mathcal G$ is an algebraic parametrization of $PB(d,\nu;n)$ it follows that $PB(d,\nu;n)$ is an unirational irreducible algebraic subset of $\mathbb{F}{\rm{ol}}\left(\nu(d+2)-2,n\right)$, $n \geq 3$. We have the following result:
\begin{theorem}
  $PB(d,\nu;n)$ is a unirational irreducible component of\break $\mathbb{F}{\rm{ol}}\left(\nu(d+2)-2,n\right);$ $n \geq 3$, $\nu\geq1$  and $d \geq 2$.
\end{theorem}
The case $\nu=1$, of linear pull-backs, was proven in \cite{caln}, whereas the case $\nu>1$, of nonlinear pull-backs, was proved in \cite{clne}. We can ask whether it is possible to obtain new families of irreducible components of nonlinear pull-back type. A natural question arises: what kind of family of rational maps should we consider? We will see that the family of rational maps we will use requires the existence of a very special set of foliations on $ \mathbb P^2$. This will be the content of theorem A.

Without details, let us state some results that we prove in this paper. Denote by $\mathcal{H}(d,2)$ the subset of foliations on $\mathbb P^2$ of degree $d$ with all singularities of Hyperbolic-Type and $Il_{3}(d,2)$ the set of foliations on $\mathbb P^2$  of degree $d$ having $3$ invariant lines in general position. Let $A\left(d\right)=Il_{3}\left(d,2\right)\cap \mathcal{H}\left(d,2\right)$. Our first result is the following:
\begin{main}\label{teoa}  Let $d\geq 2$. There exists an open and dense subset $M\left(d\right)\subset A\left(d\right)$, such that if $\mathcal G \in M\left(d\right)$  then the only algebraic invariant curves of $\mathcal G$ are the three lines. 
\end{main}
Let us describe the type of pull-back foliation that we will consider. Let $\mathcal G$ be a foliation on $\mathbb P^2$ with three invariant straight lines in general position, say $\ell_{0}, \ell_{1}$ 
and $\ell_{2}$. Consider coordinates $(X,Y,Z)\in \mathbb C^3$ such that $\ell_{0}=\Pi_2(X=0)$, $\ell_{1}=\Pi_2(Y=0)$ and $\ell_{2}=\Pi_2(Z=0)$, where $\Pi_{2}:\mathbb C^{3}\backslash \left\{0\right\} \to \mathbb P^{2}$ is the natural projection. The foliation $\mathcal G$ can be represented in these coordinates by a polynomial $1$-form of the type  $$ \Omega= YZA\left(X,Y,Z\right)dX+XZB\left(X,Y,Z\right)dY+XYC\left(X,Y,Z\right)dZ$$ where by $(1)$ $A+B+C=0$.

 Let  $f: {\mathbb P^n}  \DashedArrow[->,densely dashed    ]   {\mathbb P^2}$ be a rational map represented in the coordinates\break $(X,Y,Z)\in \mathbb{C}^3$ and $W \in \mathbb C^{n+1}$ by $\tilde f=(F^{\alpha} _{0},F^{\beta}_{1},F^{\gamma}_{2})$ where $F_{0},F_{1}$ and $F_{2} \in \mathbb C[W]$ are homogeneous polynomials without common factors satisfying $$\alpha .deg(F_{0})=\beta .deg(F_{1})=\gamma .deg(F_{2})=\nu.$$
The pull back foliation $f^{*}(\mathcal G)$ is then defined by 
$$\tilde\eta_{[f,\mathcal G]}\left(W\right)=\left[{\alpha}F_{1}F_{2}\left(A\circ F\right) dF_{0} + {\beta}F_{0}F_{2}\left(B\circ F\right) dF_{1} + {\gamma}F_{0}F_{1}\left(C\circ F\right) dF_{2}\right],$$
where each coefficient of  $\tilde\eta_{[f,\mathcal G]}\left(W\right)$ has degree $\Gamma=\nu\left[(d-1)+\frac{1}{\alpha}+\frac{1}{\beta}+\frac{1}{\gamma}\right] - 1.$ The crucial point here is that the mapping $f$ sends the three hypersurfaces $(F_{i}=0)$ contained in its critical set over the three lines invariant by $\mathcal G$.

Let $PB\left(\Gamma-1,\nu,\alpha,\beta,\gamma\right)$ be the closure in $\mathbb{F}{\rm{ol}}\left(\Gamma-1,n\right)$ of the set $\left\{ \left[\tilde\eta_{[f,\mathcal G]} \right] \right\}$, where $\tilde\eta_{[f,\mathcal G]}$ is as before. It is an unirational irreducible algebraic subset of $\mathbb{F}{\rm{ol}}\left(\Gamma-1,n\right)$. We will return to this point in Section \ref{section4}.

Let us state the main result of this work.
\begin{main}\label{teob}$PB(\Gamma-1,\nu,\alpha,\beta,\gamma)$ is a unirational irreducible component of $\mathbb{F}{\rm{ol}}\left(\Gamma-1,n\right)$ for all $n \geq 3$, $deg(F_{0}).\alpha = deg(F_{1}).\beta = deg(F_{2}).\gamma=\nu\geq2$, $(\alpha, \beta,\gamma) \in \mathbb N^{3}$ such that $1< \alpha < \beta < \gamma$ and $d \geq 2$. 
 \end{main}

\section{Foliations with 3 invariant lines}

\subsection{Basic facts} 
Denote by $I(d,2)$ the set of the holomorphic foliations on $\mathbb P^{2}$ of degree $d \geq 2$ that leaves the lines $X=0$, $Y=0$ and $Z=0$ invariant. We observe that any foliation which has $3$ invariant straight lines in general position can be carried to one of  these by a linear automorphism of $\mathbb P^{2}$.
The relation
$A+B+C=0$
enables to parametrize $I(d,2)$  as follows
\begin{align*}
  {\rm {H}}^0(\mathbb P^2,\mathcal O_{\mathbb P^2}(d-1))^{\times 2}  &\to {\rm {H}}^0(\mathbb P^2,\mathcal O_{\mathbb P^2}(d-1))^{\times 3} \\
                         (A,B) &\mapsto (A,B,-A-B).
\end{align*}
We let the group of linear automorphisms of $\mathbb P^{2}$ act on $I(d,2)$. After this procedure we obtain a set of foliations of degree $d$ that we denote by $Il_{3}(d,2)$.\par We are interested in making deformations of foliations and for our purposes we need a subset of  $Il_{3}(d,2)$ with good properties (foliations having few algebraic invariant curves and only hyperbolic singularities). We explain this properties in detail. Let $q \in U$ be an isolated singularity of a foliation $\mathcal G$ 
defined on an open subset of $U\subset \mathbb C^{2}.$ We say that $q$ is $nondegenerate$ if there exists a holomorphic vector field $X$ tangent to  $\mathcal G$ in a neighborhood of $q$ such that $DX(q)$ is nonsingular. In particular $q$ is an isolated singularity of $X.$
Let $q$ be a nondegenerate singularity of $\mathcal G$. The \textit{characteristic numbers} of $q$ are the quotients $\lambda$ and $\lambda^{-1}$ of the eingenvalues of $DX(q)$, which do not depend on the vector field $X$ chosen. If $\lambda \notin \mathbb Q_{+}$ then $\mathcal G$ exhibits exactly two (smooth and transverse) \textit{local separatrices} at $q$,  $S_{q}^{+}$ and $S_{q}^{-}$ with eigenvalues $\lambda_{q}^{+}$ and $\lambda_{q}^{-}$ and which are tangent to the characteristic directions of a vector field $X$. The characteristic numbers (also called Camacho-Sad index) of these local separatrices are given by 
$$I(\mathcal G,S_{q}^{+})=\frac{\lambda_{q}^{-}}{\lambda_{q}^{+}} \ \text{and} \
I(\mathcal G,S_{q}^{-})=\frac{\lambda_{q}^{+}}{\lambda_{q}^{-}}.$$
The singularity is \textit{hyperbolic} if the characteristic numbers are nonreal. We introduce the following spaces of foliations:

\begin{enumerate}
\item [(1)]$ND(d,2)=\{ \mathcal{G} \in \mathbb{F}{\rm{ol}}(d,2);$ the singularities of $\mathcal G$ are nondegenerate$\},$
\item[(2)] $\mathcal{H}(d,2)=\{\mathcal{G} \in ND(d,2);$ any characteristic number $\lambda$ of $\mathcal{G}$ satisfies $\lambda \in \mathbb C \backslash \mathbb R\}.$
\end{enumerate}
It is a well-known fact  \cite{ln1} that $\mathcal{H}(d,2)$ contains an open and dense subset of $\mathbb{F}{\rm{ol}}(d,2)$. Denote by $A(d)=Il_{3}(d,2)\cap \mathcal{H}(d,2).$ Observe that $A(d)$ is a Zariski dense subset of $Il_{3}(d,2)$. Concerning the set $ND(d,2)$, we have the following result, proved in \cite{ln1}.

\begin{proposition}
Let  $\mathcal G_{0} \in ND(d,2)$. Then $\#Sing(\mathcal G_{0})=d^2+d+1=N(d)$. Moreover if $Sing(\mathcal G_{0})=\{p_{1}^0,...,p_{N}^0\}$ where $p_{i}^0\neq p_{j}^0$ if $i\neq j$, then there are connected neighborhoods $U_{j} \ni p_{j}$, pairwise disjoint, and holomorphic maps $\phi_{j}:\mathcal U\subset ND(d,2) \to U_{j}$, where $\mathcal U \ni \mathcal G_{0}$ is an open neighborhood, such that for $\mathcal G \in \mathcal U$, $(Sing(\mathcal G)\cap U_{j})=\phi_{j}(\mathcal G)$  is a nondegenerate singularity.  In particular, $ND(d,2)$ is open in $\mathbb{F}{\rm{ol}}(2,d)$. Moreover, if $\mathcal G_{0} \in \mathcal{H}(d,2)$ then the two local separatrices as well as their associated eigenvalues depend analytically on $\mathcal {G}$.
\end{proposition}
\vskip0.2cm

\subsection{Proof of Theorem \ref{teoa}}

 Let us first show that there exists a holomorphic foliation in $Il_3(d,2)$ with all singularities of hyperbolic-type which does not have an algebraic invariant curve different from those three invariant straight lines. Let $U_{0}=\{\mathbb C^{2},(x,y)\}$ be an affine chart of $\mathbb P^2$. 
 Let $\mathcal H$ be a holomorphic foliation in $\mathbb P^{2}$, which is given by the polynomial vector field $X_{0}(x,y)$ on $U_{0}$ :
\begin{empheq}[left=\empheqlbrace]{align}
 \dot{x}&= x(cx + ay + \lambda), \nonumber \\\nonumber 
 \dot{y}&= y(bx +ey + \mu).
\end{empheq}
We fix $a=(1-i)$, $b=1$, $c=i$, $e=1$,  $\mu=i$ and $\lambda=1.$
It is not difficult to calculate the characteristic numbers for this foliation and after a straightforward computation we conclude that all of its singularities are hyperbolic. Using a known result (Thm.$1$ pp $891$ from \cite{cln0}) we have that if $S$ is an algebraic invariant curve for the foliation $\mathcal H$, its degree must satisfy $deg(S)\leq4$. Hence the only possible remaining algebraic invariant would be a line. On the other hand, this line would have to pass through $3$ singularities. But for the constants $a$, $b$, $c$, $e$,  $\mu$ and $\lambda$ this is impossible according to Camacho-Sad's Index Theorem \cite{ln1} for the foliation $\mathcal{H}$.  Now take the holomorphic ramified map $T: \mathbb C^{2}\to\mathbb C^{2}$ given by $(x,y)\to(x^{d-1}, y^{d-1})$, $d\geq2$. The vector field $T^{\ast}X_0$ defines a holomorphic foliation which can be naturally extended to a foliation $\mathcal{G}$ in $\mathbb{P}^2$ having three invariant lines. The map $T$ can be extended to a mapping $\mathcal{T}: \mathbb P^{2}\to\mathbb P^{2}$,  $[X:Y:Z]\to[X^{d-1}:Y^{d-1}:Z^{d-1}]$ and we have that $\mathcal G=\mathcal{T}^{\ast}\mathcal{H}$. The map $\mathcal{T}$ does not produce new algebraic invariant curves. 
To finish the argument we observe that for $d\geq2$ fixed the subsets of foliations with algebraic invariant curves different from the $3$ lines is a union of algebraic subsets whose complement in $Il_3(d,2)$ is an open and dense subset.  
For each fixed $d$ we denote this set by $M(d)$. This finishes the proof of Theorem \ref{teoa}. \cqd 

\section {Branched rational maps}

Let $f : {\mathbb P^n}  \DashedArrow[->,densely dashed    ]   {\mathbb P^2}$ be a rational map and $\tilde{f}: {\mathbb C^{n+1}} \to {\mathbb C^3}$ its natural lifting in homogeneous coordinates.

The \emph{indeterminacy locus} of $f$ is, by definition, the set $I\left(f\right)=\Pi_{n}\left(\tilde{f}^{-1}\left(0\right)\right)$. Observe that the restriction $f|_{\mathbb P^n \backslash I\left(f\right)}$ is holomorphic. We characterize the set of rational maps used throughout this text as follows:
\begin{definition} We denote by $BRM\left(n,\nu,\alpha,\beta,\gamma\right)$ the set of maps \break $\left\{f: \mathbb P^n  \DashedArrow[->,densely dashed    ]   \mathbb P^2\right\}$ of degree $\nu$ given by $f=\left(F^\alpha_{0}:F^\beta_{1}:F^\gamma_{2}\right)$  where $F_{0},F_{1}$ and $F_{2}$ are homogeneous polynomials without common factors, with $deg\left(F_{0}\right).\alpha=$ $deg\left(F_{1}\right).\beta$ $=deg\left(F_{2}\right).\gamma$ $=\nu$, $\nu \geq 2$, $\left(\alpha, \beta,\gamma\right) \in \mathbb N^{3}$ such that $1< \alpha < \beta < \gamma.$\end{definition}

Let us fix some coordinates $\left(z_{0},...,z_{n} \right)$ on $\mathbb C^{n+1}$ and $\left(X,Y,Z\right)$ on $\mathbb C^{3}$ and denote by $\left(F^\alpha_{0},F^\beta_{1},F^\gamma_{2}\right)$ the components of $f$ relative to these coordinates. Let us note that the indeterminacy locus $I(f)$ is the intersection of the $3$ hypersurfaces $(F_{0}=0)$, $(F_{1}=0)$ and $(F_{2}=0)$.

\begin{definition}\label{generic} We say that $f  \in BRM\left(n,\nu,\alpha,\beta,\gamma\right)$ is $generic$ if for all $p \in$ $\tilde{f}^{-1}\left(0\right)\backslash\left\{0\right\}$ we have $dF_{0}\left(p\right)\wedge dF_{1}\left(p\right)\wedge dF_{2}\left(p\right) \neq 0.$  
\end{definition}

This is equivalent to saying that $f  \in BRM\left(n,\nu,\alpha,\beta,\gamma\right)$ is $generic$ if $I(f)$ is the transverse intersection of the $3$ hypersurfaces $(F_{0}=0)$, $(F_{1}=0)$ and $(F_{2}=0)$.  As a consequence we have that  the set $I(f)$ is smooth. For instance, if $n=3$, $f$ is generic and $deg(f)=\nu$, then by Bezout's theorem $I\left(f\right)$ consists of $\frac{\nu^{3}}{\alpha \beta \gamma}$ distinct points with multiplicity ${\alpha \beta \gamma}$. If $n=4$, then $I\left(f\right)$ is a smooth connected algebraic curve in $\mathbb P^4$ of degree $\frac{\nu^{3}}{\alpha \beta \gamma}$. In general, for $n \geq 4$, $I\left(f\right)$ is a smooth connected algebraic submanifold of $\mathbb P^n$ of degree $\frac{\nu^{3}}{\alpha \beta \gamma}$ and codimension three. 

Denote $\nabla F_{k}= (\frac{\partial F_{k}}{\partial z_{0}},...,\frac{\partial F_{k}}{\partial z_{n}})$. Consider the derivative matrix 
 $$M=\begin{bmatrix}
 \alpha\left(F_{0}^{\alpha-1}\right)\nabla F_{0} \\ \beta\left(F_{1}^{\beta-1}\right)\nabla F_{1} \\  \gamma\left(F_{2}^{\gamma-1}\right)\nabla F_{2}  
 \end{bmatrix}.$$ 
The critical set of $\tilde{f}$ is given by the points of $ \mathbb C^{n+1}\backslash \ {0}$ where rank$(M)\leq3$; it is the union of two sets. The first is given by   the set of $\left\{ Z \in \mathbb C^{n+1}\backslash \ {0} \right\}=X_{1}$ such that the rank of the following matrix 
 $$N=\begin{bmatrix}
 \nabla F_{0} \\ \nabla F_{1} \\ \nabla F_{2}  
 \end{bmatrix}$$ is smaller than $3.$ The second is the subset
$$X_2=\left\{ Z \in \mathbb C^{n+1}\backslash \left\{0\right\} | \left(F_{0}^{\alpha-1} . F_{1}^{\beta-1} . F_{2}^{\gamma-1}\right)\left(Z\right)=0 \right\}.$$

Denote $P\left(f\right)=\Pi_{n}\left(X_{1} \cup X_{2}\right)$. The set of generic maps will be denoted by $Gen\left(n,\nu,\alpha,\beta,\gamma\right)$. We state the following result whose proof is standard in algebraic geometry:
\begin{proposition} $Gen\left(n,\nu,\alpha,\beta,\gamma\right)$ is a Zariski dense subset of \break$BRM\left(n,\nu,\alpha,\beta,\gamma\right)$.
\end{proposition}

\section{Ramified pull-back components - Generic conditions}\label{section4}

Let us fix a coordinate system $(X,Y,Z)$ on $\mathbb P^2$ and denote by $\ell_{0}$, $\ell_{1}$ and $\ell_{2}$ the straight lines that correspond to the planes $X=0$, $Y=0$ and $Z=0$ in $\mathbb C^3$, respectively.
Let us denote by $\tilde M\left(d\right)$ the subset $M\left(d\right)\cap I(d,2)$.
\begin{definition} Let $f \in Gen\left(n,\nu,\alpha,\beta,\gamma\right)$. We say that  
$\mathcal G \in M\left(d\right)$ is in generic position with respect to $f$ if $\left[Sing\left(\mathcal G\right)\cap Y_{2}\right]=\emptyset,$ where 
$$Y_{2}(f)=Y_{2}:=\displaystyle\Pi_2\left[\tilde{f}\left\{w \in \mathbb C^{n+1} |dF_{0}\left(w\right)\wedge dF_{1}\left(w\right)\wedge dF_{2}\left(w\right) = 0\right\}\right]$$ and $\ell_{0}$, $\ell_{1}$ and $\ell_{2}$ are $\mathcal G$-invariant.
\end{definition}

In this case we say that $\left(f,\mathcal{G}\right)$ is a generic pair. In particular, when we fix a map $f\in Gen(n,\nu,\alpha,\beta,\gamma)$ the set $\mathcal{A}=\left\{\mathcal{G} \in M\left(d\right) | Sing\left(\mathcal G\right) \cap Y_{2}(f)= \emptyset \right\}$ is an open and dense subset in $M(d)$ \cite{lnsc}, since $VC(f)$ {is an algebraic curve in}  ${\mathbb P^2}.$ The set $U_{1}:=\{ (f,\mathcal G) \in Gen(n,\nu,\alpha,\beta,\gamma)\times\tilde M\left(d\right)| Sing\left(\mathcal G\right)\cap Y_{2}(f)= \emptyset \}$ is an open and dense subset of $Gen(n,\nu,\alpha,\beta,\gamma)\times\tilde M\left(d\right)$. Hence the set \break $\mathcal W:=\left\{\tilde\eta_{[f,\mathcal G]}| \left(f,\mathcal G\right)\in U_{1}\right\}$ is an open and dense subset of $PB\left(\Gamma-1,\nu,\alpha,\beta,\gamma\right)$.
  
The following result, concerning the degree of a foliation given by a generic pair, is proved in the Appendix.
    \begin{proposition} \label{graupargenerico}If $\mathcal F$ comes from a generic pair, then the degree of $\mathcal F$ is $$\nu\left[\left(d-1\right)+\frac{1}{\alpha}+\frac{1}{\beta}+\frac{1}{\gamma}\right] - 2.$$ 
   \end{proposition}

Consider the set of foliations $Il_{3}\left(d,2\right)$, $d\geq 2,$ and the following map:
\begin{eqnarray*}
\Phi:BRM\left(n,\nu,\alpha,\beta,\gamma\right) \times Il_{3}\left(d,2\right) &\to&\mathbb{F}{\rm{ol}}\left(\Gamma-1,n\right)\\
\left(f,\mathcal G\right) &\to& f ^{\ast}\left(\mathcal G\right)=\Phi\left(f,\mathcal G\right).
\end{eqnarray*}  
The image of $\Phi$ can be written as:  $$\Phi\left(f,\mathcal G\right)=\left[{\alpha}F_{1}F_{2}\left(A\circ F\right) dF_{0} + {\beta}F_{0}F_{2}\left(B\circ F\right) dF_{1} + {\gamma}F_{0}F_{1}\left(C\circ F\right) dF_{2}\right].$$ Recall that $\Phi\left(f,\mathcal G\right)=\tilde\eta_{[f,\mathcal G]}$. More precisely, let $PB(\Gamma-1,\nu,\alpha,\beta,\gamma)$ be the closure in $\mathbb{F}{\rm{ol}}\left(\Gamma-1,n\right)$ of the set of foliations $\mathcal{F}$ of the form $f^*\left(\mathcal{G}\right)$, where $f\in BRM\left(n,\nu,\alpha,\beta,\gamma\right)$ and $\mathcal G \in Il_{3}(2,d).$ Since $BRM\left(n,\nu,\alpha,\beta,\gamma\right)$ and $Il_{3}(2,d)$ are irreducible algebraic sets and the map $\left(f,\mathcal{G}\right) \to f^*\left(\mathcal{G}\right) \in \mathbb{F}{\rm{ol}}\left(\Gamma-1,n\right)$ is an algebraic parametrization of $PB(\Gamma-1,\nu,\alpha,\beta,\gamma)$, we have that $PB(\Gamma-1,\nu,\alpha,\beta,\gamma)$ is an irreducible algebraic subset of $\mathbb{F}{\rm{ol}}\left(\Gamma-1,n\right)$. Moreover, the set of generic pull-back foliations $\left\{\mathcal{F}; \mathcal{F} = f^*(\mathcal{G}),\ \text{where} \ \left(f,\mathcal{G}\right)\right.$ is a generic pair$\left.\right\}$ is an open (not Zariski) and dense subset of $PB(\Gamma-1,\nu,\alpha,\beta,\gamma)$ for $\nu \geq 2$, $\left(\alpha, \beta,\gamma\right) \in \mathbb N^{3}$ such that $1< \alpha < \beta < \gamma$ and $d\geq2$. 

\section{Description of generic ramified pull-back foliations on $\mathbb P^n$}
\subsection{The Kupka set of $\mathcal{F} = f^*(\mathcal{G})$}\label{section5.1}

Let $\tau$ be a singularity of $\mathcal{G}$ and $V_{\tau}=\overline{f^{-1}(\tau)}$.  If $(f,\mathcal G)$ is a generic pair then $V_{\tau}\backslash I(f)$ is contained in the Kupka set  of $\mathcal F$. As an example we detail the case where $\tau$ is a corner, say $a=[0:0:1]$. Fix $p\in V_{\tau}\backslash I(f)$. There exist local analytic coordinate systems such that $f(x,y,z)=(x^{\alpha},y^{\beta})=(u,v)$. Suppose that $\mathcal {G}$ is represented by the 1-form $\omega$; the hypothesis of $\mathcal G$ being of Hyperbolic-type implies that we can suppose $\omega(u,v) =  \lambda_{1}u(1+R(u,v))dv - \lambda_{2}vdu,$ where $\frac{\lambda_{2}}{\lambda_{1}} \in \mathbb{C}\backslash \mathbb{R}$. We obtain $\tilde \omega(x,y)=f^{\ast}(\omega)=(x^{\alpha -1}.y^{\beta -1})(\lambda_{1}\beta x(1+R(x^{\alpha},y^{\beta})dy-\alpha\lambda_{2}ydx)=(x^{\alpha -1}.y^{\beta -1})\hat \omega(x,y)$ and so $d\hat \omega(p)\neq0$. Therefore if $p$ is as before it belongs to the Kupka-set of $\mathcal {F}$. For the other points the argumentation is analogous. This is the well known Kupka-Reeb phenomenon, and we say that $p$ is contained in the Kupka-set of $\mathcal {F}$. It is known that this local product structure is stable under small perturbations of $\mathcal F$ \cite{kupka},\cite{gln}. 
 
\subsection{Generalized Kupka and quasi-homogeneous singularities}\label{section5.2}
 In this section we will recall the quasi-homogeneous singularities of an  integrable holomorphic $1$-form.  They appear in the indeterminacy set of $f$ and play a central role in great part of the proof of Theorem  B. 
\begin{definition} Let  $\omega$ be an holomorphic  integrable 1-form defined in a neighborhood of $p \in {\mathbb{C}^3}$. We say that $p$ is a Generalized Kupka(GK) singularity of $\omega$ if $\omega(p)=0$ and either 
$d\omega(p)\neq0$ or $p$ is an isolated zero of $d\omega$.   
\end{definition}
 \par
Let $\omega$ be an integrable $1$-form in a neighborhood of $p  \in {\mathbb{C}^3}$  and $\mu$ be a holomorphic $3$-form such that $\mu({p}) \neq 0$. Then $d\omega=i_{\mathcal Z}(\mu)$ where $\mathcal Z$ is a holomorphic vector field. 
\begin{definition} We say that $p$ is a quasi-homogeneous singularity of $\omega$ if  $p$ is an isolated singularity of $\mathcal Z$ and the germ of $\mathcal Z$ at $p$ is nilpotent, that is, if $L=D\mathcal Z(p)$ then all eigenvalues of $L$ are equals to zero.
\end{definition}
 This definition is justified by the following result that can be found in \cite{ln1} or \cite{ccgl}:
\begin{theorem} \label{teo5.3}
Let $p$ be a quasi-homogeneous singularity of an holomorphic integrable 1-form $\omega.$  Then there exists two holomorphic vector fields $S$ and $\mathcal Z$ and a local chart  $U:=(x_{0},x_{1},x_{2})$  around $p$ such that  $x_{0}(p)=x_{1}(p)=x_{2}(p)=0$ and:
\begin{enumerate}
\item[(a)] $\omega$ = $\lambda i_{S}i_{\mathcal Z}(dx_0 \wedge dx_{1} \wedge dx_{2}),$ $\lambda \in \mathbb Q_{+}$ 
$d\omega= i_{\mathcal Z}(dx_{0} \wedge dx_{1} \wedge dx_{2})$ and $\mathcal Z=(rot (\omega))$;
\item[(b)] $S=p_{0}x_{0} \frac{\partial}{\partial x_{0}}+ p_{1}x_{1} \frac{\partial}{\partial x_{1}}+p_{2}x_{2} \frac{\partial}{\partial x_{2}} $, where, $p_{0}$, $p_{1}$, $p_{2}$ are positive integers with $g.c.d (p_{0},p_{1},p_{2})=1$;
\item[(c)]  $p$ is an isolated singularity for $\mathcal Z$, $\mathcal Z$ is polynomial in the chart\break $U:=(x_{0},x_{1},x_{2})$ and  $[S,\mathcal Z]=\ell \mathcal Z$, where $\ell \geq1$.
\end{enumerate}
\end{theorem}
 \begin{definition}
 Let $p$ be a quasi-homogeneous singularity of $\omega.$ We say that it is of the type $(p_{0}:p_{1}:p_{2};\ell)$, if  for some local chart and vector fields $S$ and $\mathcal Z$  the properties $(a),(b)$ and $(c)$ of the {\textrm{Theorem}} \ref{teo5.3} are satisfied.
 \end{definition}

 We can now state the stability result, whose proof can be found in\break \cite{ccgl}:
 \begin{proposition}\label{prop5.5} Let $(\omega_{s})_{s \in \Sigma}$ be a holomorphic family of integrable $1$-forms defined in a neighborhood of a compact ball  $B= \{ {z \in {\mathbb{C}^3} ; |z|} \leq \rho \}$, where $\Sigma$ is a neighborhood of $0 \in {\mathbb{C}^k}.$  Suppose that all singularities of $\omega_{0}$ in $B$ are $GK$ and that $sing(d\omega_{0})\subset int(B)$. Then there exists $\epsilon >0$ such that if $s \in B(0,\epsilon)\subset\Sigma,$ then all singularities of $\omega_{s}$ in $B$ are GK. Moreover, if $0 \in B$ is a quasi-homogeneous singularity of type $(p_{0}:p_{1}:p_{2};\ell)$ then there exists a holomorphic map $B(0,\epsilon) \ni s \mapsto z(s)$, such that $z(0)=0$ and $z(s)$ is a $GK$ singularity of  $\omega_{s}$ of the same type (quasi-homogeneous of the type $(p_{0}:p_{1}:p_{2};\ell)$, according to the case).
 \end{proposition}
 
 Let us describe  $\mathcal{F} = f^*(\mathcal{G})$ in a neighborhood of a point $p \in I(f).$
\noindent It is easy to show that there exists a local chart $(U,(x_0,x_1,x_2,y)\in \mathbb C^3\times \mathbb C^{n-2})$  around $p$ such that the lifting $\tilde f$ of $f$ is of the form $\tilde f|_{U}=(x^{\alpha} _{0},x^{\beta} _{1},x^{\gamma} _{2}):U \to {\mathbb{C}^3}$. In particular $\mathcal {F}|_{U(p)}$ is represented by the $1$-form 
\begin{equation}\label{eta}
\eta (x_{0},x_{1},x_{2},y) = \alpha. x_{1.}x_{2}.A(x^{\alpha} _{0},x^{\beta} _{1},x^{\gamma} _{2})dx_{0} 
+ \beta. x_{0.}x_{2}.B(x^{\alpha} _{0},x^{\beta} _{1},x^{\gamma} _{2})dx_{1}\end{equation}\begin{equation*} + \gamma. x_{0.}x_{1}.C(x^{\alpha} _{0},x^{\beta} _{1},x^{\gamma} _{2})dx_{2}. 
\end{equation*}

Let us now obtain the vector field $S$ as in Theorem \ref{teo5.3}. Consider the radial vector field $R=X \frac{\partial}{\partial X}+ Y \frac{\partial}{\partial Y}+ Z \frac{\partial}{\partial Z}$. Note that in the coordinate system above it transforms into
$$ \frac{1}{\alpha}x_{0}  \frac{\partial}{\partial x_{0}}+\frac{1}{\beta}x_{1}  \frac{\partial}{\partial x_{1}}+\frac{1}{\gamma}x_{2}  \frac{\partial}{\partial x_{2}}.$$ Since the eigenvalues of $S$ have to be integers, after a multiplication by ${\alpha\beta\gamma}$ we obtain $$S=({\beta}{\gamma}) x_{0} \frac{\partial}{\partial x_{0}}+ ({\alpha}{\gamma})x_{1} \frac{\partial}{\partial x_{1}}+({\alpha}{\beta}) x_{2} \frac{\partial}{\partial x_{2}}.$$

\begin{remark}
We observe that if g.c.d$(\beta\gamma,\alpha\gamma,\alpha\beta)=\theta\neq1$ we replace $(\beta\gamma,\alpha\gamma,\alpha\beta)$ by $\frac{(\beta\gamma,\alpha\gamma,\alpha\beta)}{\theta}$ and we repeat with minor modifications the same arguments. Hence we can suppose for simplicity that g.c.d$(\beta\gamma,\alpha\gamma,\alpha\beta)=1$. 
  \end{remark}
\begin{lemma} \label{lema5.6}If $p\in I(f)$ then $p$ is a quasi-homogeneous singularity of $\eta$.    \end{lemma} 
  \begin{proof}
First of all note that $i_{S}\eta=0$.
Let us calculate $L_{S}\eta$, where $L$ denotes the Lie derivative. By standard computations we have that $L_{S}\eta=m\eta,$ where $m=[(\beta\gamma+\alpha\gamma+\alpha\beta)+(\alpha\beta\gamma)(d-1)]$. This implies that the singular set of $\eta$ is invariant under the flow of $S$.  
The vector field $\mathcal Z$ such that $\eta=i_{S}i_{\mathcal Z}(dx_0 \wedge dx_{1} \wedge dx_{2})$ is given by
$$\mathcal Z=\mathcal Z_{0}(x_{0},x_{1},x_{2}) \frac{\partial}{\partial x_{0}}+ \mathcal Z_{1}(x_{0},x_{1},x_{2}) \frac{\partial}{\partial x_{1}}+ \mathcal Z_{2}(x_{0},x_{1},x_{2}) \frac{\partial}{\partial x_{2}}$$
where $\mathcal Z_{i} (x_{0},x_{1},x_{2})= x_{i}.\tilde{A}_{i}(x^{\alpha} _{0},x^{\beta} _{1},x^{\gamma} _{2})$.
The polynomials $\tilde A_{i}(X,Y,Z)$ are homogeneous of degree $(d-1)$ and they are not unique. We must show that the origin is an isolated singularity of $\mathcal Z$ and all eigenvalues of $D\mathcal Z(0)$ are $0$. By straightforward computation we find that the Jacobian matrix $D\mathcal Z(0)$ is the null matrix, hence all its eigenvalues are null. Since all singular curves of $\mathcal F$ in a neighborhood $(U,(x_0,x_1,x_2))$ of $0$ are of Kupka type, as proved in Section \ref{section5.1}, it follows that the origin is an isolated sigularity of $\mathcal Z$. Note that the unique singularities of $\eta$ in the neighborhood $(U,(x_0,x_1,x_2))$ of $0$ come from $\tilde f^\ast Sing(\mathcal G)$; this follows from the fact that  $Sing\left(\mathcal G\right) \cap (VC(f) \backslash \ell_{0} \cup \ell_{1} \cup \ell_{2})= \emptyset$.
  On the other hand we have seen that $(f)^{-1}(sing(\mathcal G))\backslash I(f)$ is contained in the Kupka set of $\mathcal{F}$. Hence the point $p$  is an isolated singularity of $d\eta$ and thus an isolated singularity of $\mathcal Z$.
 \end{proof} 
 
 As a consequence, in the case $n=3$ any $p\in I(f)$ is a quasi-homogeneous singularity of type $\left [{\beta}{\gamma}:{\alpha}{\gamma}:{\alpha}{\beta}  \right ]$. In the case $n\geq4$ the argument is analogous. Moreover, in this case there will be a local structure product near any point $p\in I(f)$. In fact in the case $n\geq4$ we have:
\begin{cor} Let $(f,\mathcal G)$ be a generic pair. Let $p \in I(f)$ and $\eta$ an 1-form defining $\mathcal F$ in a neighborhood of $p$. Then there exists a 3-plane $\Pi \subset  \mathbb C^{n}$ such that $d(\eta)|_{\Pi}$ has an isolated singularity at $0 \in{\Pi}. $
\end{cor}
\begin{proof} Immediate from the local product structure.
\end{proof}

\subsection{Deformations of the singular set of ${\mathcal F}_0= f_{0}^\ast (\mathcal G_{0})$}
In this section we give some auxiliary lemmas which assist in the proof of Theorem \ref{teob}. We have constructed an open and dense subset $\mathcal W$ inside {$PB(\Gamma-1,\nu,\alpha,\beta,\gamma)$} containing the generic pull-back foliations.  We will show that for any rational foliation $\mathcal{F}_{0} \in \mathcal W$ and any germ of a holomorphic family of foliations $(\mathcal{F}_{t})_{t \in (\mathbb{C},0)}$ such that $\mathcal{F}_{0}=\mathcal{F}_{t=0}$ we have $\mathcal{F}_{t} \in PB(\Gamma-1,\nu,\alpha,\beta,\gamma)$ for all $t \in (\mathbb{C},0).$
\begin{lemma} \label{Lemma5.8} There exists a germ of isotopy of class $C^{\infty}$, $(I(t))_{t \in (\mathbb {C},0)}$ having the following properties: 
   \begin{enumerate}
\item[(i)]  $I(0)=I(f_{0})$ and $I(t)$ is algebraic and smooth of codimension $3$ for all  ${t \in ({\mathbb {C}}, 0)}.$
 \item[(ii)] For all $p \in I(t)$, there exists a neighborhood $U(p,t)=U$ of $p$ such that $\mathcal{F}_{t}$ is equivalent to the product of a regular foliation of codimension $3$ and a singular foliation $\mathcal F_{p,t}$ of codimension one given by the 1-form $\eta_{p,t}$.
 \end{enumerate}
 \begin{remark}
The family of $1$-forms $\eta_{p,t}$, represents the quasi-homogeneous foliation given by the Proposition \ref{prop5.5}. 
  \end{remark}
   \begin{proof} See \cite[lema 2.3.2, p.81]{ln}.
 \end{proof}

\end{lemma}

\begin{remark}
In the case $n>3$, the variety $I(t)$ is connected since $I(f_{0})$ is connected. The local product structure in $I(t)$ implies that the transversal type of $\mathcal{F}_{t}$ is constant. In particular, $\mathcal F_{p,t}$, does not depend on $p \in I(t)$. In the case $n=3$, $I(t)=p_{1}(t),..., p_{j}(t),...,p_{\frac{\nu^{3}}{\alpha \beta \gamma}}(t)$ and we can not guarantee a priori that $\mathcal F_{p_{i},t}=\mathcal F_{p_{j},t}$, if $i \neq j.$
\end{remark}
  The singular set of  $\mathcal{G}_{0}$ consists of the points $a=[0:0:1]$, $b=[0:1:0]$, $c=[1:0:0]$, and the subsets $\mathcal S_{W}(\mathcal G_{0})$,  $\mathcal S_{\ell_{r}}(\mathcal G_{0})$, $0\leq r\leq2$.  We know that  $\# \mathcal S_{W}(\mathcal G_{0})=(d-1)^{2}$, $\# \mathcal S_{\ell_{r}}(\mathcal G_{0})=(d-1)$, $0\leq r\leq2$.  Let $\tau \in Sing( \mathcal G_{0})$ and $K(\mathcal F_0)=\cup_{{\tau\in Sing( \mathcal G_{0})}}V_{\tau}\backslash I(f_{0})$ where $V_{\tau}=\overline{f_{0}^{-1}(\tau)}.$ 
As in Lemma \ref{Lemma5.8}, let us consider a representative of the germ $(\mathcal F_{t})_{t}$, defined on a disc $D_{\delta}:=(|t|<\delta).$ 
  
\begin{lemma} 
There exist $\epsilon >0$ and smooth isotopies $\phi_{\tau}:D_{\epsilon}\times V_{\tau} \to \mathbb P^n, \tau \in Sing(\mathcal G_{0})$, such that $V_{\tau}(t)=\phi_{\tau}(\{t\}\times V_{\tau})$ satisfies:
\begin{enumerate}
\item[(a)] $V_{\tau}(t)$ is an algebraic subvariety of codimension two of $\mathbb P^n$ and  $V_{\tau}(0) = V_{\tau}$ for all $ \tau \in Sing(\mathcal G_{0})$ and for all $t \in D_{\epsilon}.$
\item[(b)] $I (t) \subset V_{\tau}(t)$ for all $\tau \in Sing(\mathcal G_{0})$ and for all $t \in D_{\epsilon} $. Moreover, if $\tau \neq \tau'$, and $\tau, \tau' \in Sing(\mathcal G_{0})$,  we have $V_{\tau}(t)  \cap V_{\tau'}(t) = I(t)$ for all $t \in D_{\epsilon}$ and the intersection is transversal.
\item[(c)]   $V_{\tau}(t) \backslash I(t)$ is contained in the Kupka-set of $\mathcal{F}_t$ for all $ \tau \in Sing(\mathcal G_{0})$ and for all $t \in D_{\epsilon}.$ In particular, the transversal type of $\mathcal F_{t}$ is constant along $V_{\tau}(t) \backslash I(t)$. 
\end{enumerate}
\end{lemma}
\begin{proof}  See \cite[lema 2.3.3, p.83]{ln}.
 \end{proof}

\subsection{End of the proof of Theorem \ref{teob}}\label{subsection5.4} 

We divide the end of the proof of Theorem \ref{teob} in two parts. In the first part we construct a family  of rational maps $f_{t}:{\mathbb{P}^n \DashedArrow[->,densely dashed    ]   \mathbb{P}^2}$, $f_{t} \in Gen(n,\nu,\alpha,\beta,\gamma)$, such that $(f_{t})_{t \in D_{\epsilon}}$ is a deformation of $f_{0}$ and the subvarieties $V_\tau$, $\tau \in Sing \mathcal{G}_0$, are fibers of $f_t$ for all $t$. In the second part we show that there exists a family of foliations $(\mathcal {G}_{t})_{t\in D_{\epsilon}}, \mathcal {G}_{t} \in  \mathcal A$ (see Section \ref{section4}) such that  $\mathcal{F}_{t}= f_{t}^{ \ast}(\mathcal{G}_{t})$ for all $t\in D_{\epsilon}$. 

\subsubsection{Part 1}\label{part1} Let us define the family of candidates that will be a deformation of the mapping $f_{0}.$   Set $V_{a}=\overline{{f}_{0}^{-1}(a)}$, $V_{b}=\overline{{f}_{0}^{-1}(b)}$, $V_{c}=\overline{{f}_{0}^{-1}(c)}$, where $a=[0:0:1]$, $b=[0:1:0]$ and  $c=[1:0:0]$ and denote by $V_{\tau^\ast}=\overline{{f}_{0}^{-1}(\tau^\ast)}$, where $\tau^\ast \in Sing(\mathcal G_{0})\backslash\{a,b,c\}$.
  \begin{proposition}\label{recupmapas}
Let $(\mathcal{F}_{t})_{t \in D_{\epsilon}}$ be a deformation of $\mathcal F_{0}= f_{0}^*(\mathcal G_{0})$, where $(f_{0}, \mathcal G_{0})$ is a generic pair, with $\mathcal G_{0} \in \mathcal A$, $f_0 \in Gen\left(n,\nu,\alpha,\beta,\gamma\right)$ and $deg(f_{0})=\nu\geq 2$. Then there exists a deformation $({f}_{t})_{t \in D_{\epsilon}}$ of $f_{0}$ in $Gen\left(n,\nu,\alpha,\beta,\gamma\right)$ such that:
\begin{enumerate}
\item[(i)]$V_{a}(t),V_{b}(t)$ and $V_{c}(t)$ are fibers of $({f}_{t})_{t \in D_{\epsilon'}}$.
\item[(ii)] $I(t)=I(f_{t}), \forall {t \in D_{\epsilon'}}$.
    \end{enumerate}
  \end{proposition}  
\begin{proof} Let $\tilde f_{0}=(F^\alpha_{0},F^\beta_{1},F^\gamma_{2}): \mathbb C^{n+1} \to\mathbb C^{3}$ be the homogeneous expression of $f_{0}$. Then $V_{c}$, $V_{b}$, and $V_{a}$ appear as the complete intersections $(F_{1}=F_{2}=0)$, $(F_{0}=F_{2}=0)$, and $(F_{0}=F_{1}=0)$ respectively. Hence $I(f_{0})=V_{a} \cap V_{b} =V_{a} \cap V_{c}= V_{b} \cap V_{c}$. It follows from \cite{Ser0} (see section 4.6 pp 235-236) that $V_{a}(t)$ is a complete intersection, say $V_{a}(t)=(F_{0}(t)=F_{1}(t)=0)$, where $(F_{0}(t))_{t \in D_{\epsilon'}}$ and $(F_{1}(t))_{t \in D_{\epsilon'}}$ are deformations of $F_{0}$ and $F_{1}$ and $ D_{\epsilon'}$ is a possibly smaller neighborhood of $0$. Moreover, $F_{0}(t)=0$ and $F_{1}(t)=0$ meet transversely along $V_{a}(t)$. In the same way, it is possible to define $V_{c}(t)$ and $V_{b}(t)$ as complete intersections, say $(\hat F_{1}(t)=F_{2}(t)=0)$ and $(\hat F_{0}(t)=\hat F_{2}(t)=0)$ respectivelly, where $(F_{j}(t))_{t \in D_{\epsilon'}}$ and $(\hat F_{j}(t))_{t \in D_{\epsilon'}}$ are deformations of $F_{j}$, $0\leq j\leq2$. 

We will prove we can find polynomials $P_{0}(t)$, $P_{1}(t)$ and $P_{2}(t)$ such that $V_{c}(t)=(P_{1}(t)=P_{2}(t)=0)$, $V_{b}(t)=(P_{0}(t)=P_{2}(t)=0)$ and $V_{a}(t)=(P_{0}(t)=P_{1}(t)=0)$. Observe first that since $F_{0}(t), F_{1}(t)$ and $F_{2}(t)$ are near $F_{0}$, $F_{1}$ and $F_{2}$ respectively, they meet as a regular complete intersection at:
$$J(t)=(F_{0}(t)=F_{1}(t)=F_{2}(t)=0) = V_a(t) \cap (F_2(t)=0).$$
Hence $J(t) \cap (\hat F_{1}(t)=0)=V_{c}(t)\cap V_{a}(t)=I(t)$, which implies that $I(t)\subset J(t).$  Since $I(t)$ and $J(t)$ have $\frac{\nu^{3}}{\alpha \beta \gamma}$ points, we have that  $I(t)=J(t)$ for all $t \in D_{\epsilon'}$. 
\begin{remark}In the case $n\geq 4$, both sets are codimension-three smooth and connected submanifolds of $\mathbb P^n$, implying again that $I(t)=J(t).$ In particular, we obtain that $$I(t)=(F_{0}(t)=F_{1}(t)=F_{2}(t)=0)\subset (\hat F_{j}(t)=0), 0\leq j\leq2.$$
 \end{remark}  
 
We will use the following version of Noether's Normalization Theorem (see \cite{ln} p 86):

\begin{lemma} (Noether's Theorem) Let $G_{0},...,G_{k}  \in \mathbb {C}[z_{1},...,z_{m}]$ be homogeneous polynomials where $0 \leq k \leq m$ and $m \geq 2$, and $X=(G_{0}=...=G_{k}=0)$. Suppose that the set $Y:=\{p \in X | dG_{0}(p) \wedge...\wedge dG_{k}(p)=0\}$ is either $0$ or $\emptyset$. If $G  \in \mathbb {C}[z_{1},...,z_{m}]$ satisfies $G|_{X} \equiv 0$, then  $G$  $\in$  $<G_{0},...,G_{k}>$. 
\end{lemma} 
 
Take $k=2$, $G_{0}=F_{0}(t)$, $G_{1}=F_{1}(t)$ and $G_{2}=F_{2}(t)$. Using Noether's Theorem with $Y=0$ and the fact that all polynomials involved are homogeneous, we have $\hat F_{1}(t)$ $\in$ $<F_{0}(t),F_{1}(t),F_{2}(t)>$. Since $deg(F_{0}(t))>deg(F_{1}(t))>deg(F_{2}(t))$, we conclude that $\hat F_{1}(t)=F_{1}(t)+g(t)F_{2}(t)$, where $g(t)$ is a homogeneous polynomial of degree $deg(F_{1}(t))-deg(F_{2}(t))$. Moreover observe that $V_{c}(t)=V(\hat F_{1}(t),F_{2}(t))=V(F_{1}(t),F_{2}(t))$, where $V(H_{1},H_{2})$ denotes the projective algebraic variety defined by $(H_{1}=H_{2}=0)$. Similarly for $V_{b}(t)$ we have that $\hat F_{2}(t)$ $\in$ $<F_{0}(t),F_{1}(t),F_{2}(t)>$.  On the other hand, since $\hat F_{2}(t)$ has the lowest degree, we can assume that $\hat F_{2}(t)=F_{2}(t)$.
   
In an analogous way we have that $\hat F_{0}(t)=F_{0}(t)+m(t)F_{1}(t)+n(t)F_{2}(t)$ for the polynomial $\hat F_{0}(t)$. Now observe that $V(\hat F_{0}(t),\hat F_{2}(t))=V(F_{0}(t)+m(t)F_{1}(t),F_{2}(t))$. Hence we can define $f_{t}=(P_{0}^{\alpha}(t),P_{1}^{\beta}(t),P_{2}^{\gamma}(t))$
 where  $P_{0}(t)=F_{0}(t)+m(t)F_{1}(t)$, $P_{1}(t)=F_{1}(t)$ and $P_{2}(t)=F_{2}(t)$. This defines a  family of mappings $({f}_{t})_{t \in D_{\epsilon'}}:\mathbb P^{3} \DashedArrow[->,densely dashed    ]\mathbb P^{2}$, and $V_{a}(t)$, $V_{b}(t)$ and $V_{c}(t)$ are fibers of ${f}_{t}$ for fixed $t$. Observe that, for ${\epsilon'}$ sufficiently small, $({f}_{t})_{t \in D_{\epsilon'}}$ is generic in the sense of definition $3.2$, and its indeterminacy locus $I({f}_{t})$ is precisely $I(t).$ Moreover, since $Gen(3,\nu,\alpha,\beta,\gamma)$ is open, we can suppose that this family $({f}_{t})_{t \in D_{\epsilon'}}$ is in $Gen\left(3,\nu,\alpha,\beta,\gamma\right)$. This concludes the proof of proposition $5.10$. 
\end{proof}

Now we will prove that the remaining curves $V_{\tau}(t)$ are also fibers of $f_{t}$. In the local coordinates $X(t)=(x_{0}(t),x_{1}(t),x_{2}(t))$ near some point of $I(t)$ we have that the vector field $S$ is diagonal and the components of the map ${f}_{t}$ are written as follows:
\begin{equation}\label{eq3}
P_{0}(t)=u_{0t}x_{0}(t)+x_{1}(t)x_{2}(t)h_{0t}
\end{equation}
\begin{equation*}
P_{1}(t)=u_{1t}x_{1}(t)+x_{0}(t)x_{2}(t)h_{1t}
\end{equation*}
\begin{equation*}
P_{2}(t)=u_{2t}x_{2}(t)+x_{0}(t)x_{1}(t)h_{2t}
\end{equation*}
where the functions $u_{it} \in \mathcal O^{*}(\mathbb C^3,0)$ and $h_{it} \in \mathcal O(\mathbb C^3,0) ,0\leq i \leq 2$. Note that when the parameter $t$ goes to $0$ the functions $h_{i}(t),0\leq i \leq 2$ also goes to $0$.
We want to show that an orbit of the vector field $S$ in the coordinate system $X(t)$ that extends globally like a singular curve of the foliation $\mathcal F_{t}$ is a fiber of ${f}_{t}$. 

The condition $\alpha<\beta<\gamma$ implies that $\alpha\gamma<\beta(\alpha+\gamma)$ and $\alpha\beta<\gamma(\alpha+\beta)$. Firstly we prove that if $\beta\gamma\leq\alpha(\beta+\gamma)$ then any generic orbit of the vector field $S$ that extends globally as singular curve of the foliations $\mathcal F_{t}$ is also a fiber of $f_{t}$ for fixed $t$. Afterwords we show that if we have $\beta\gamma>\alpha(\beta+\gamma)$ then any orbit of the vector field $S$ contained in the coordinate planes and which extends globally as singular curve of the foliations $\mathcal F_{t}$ are fibers of the mapping $f_{t}$. Using these facts, we can prove that any generic orbit of the vector field $S$ that extends globally as singular curve of the foliations $\mathcal F_{t}$ is also a fiber of $f_{t}$ in the second case.

\begin{lemma}\label{casomenor} If $\beta\gamma\leq\alpha(\beta+\gamma)$ then any generic orbit of the vector field $S$ that extends globally as singular curve of the foliations $\mathcal F_{t}$ is also a fiber of $f_{t}$ for fixed $t$.
\end{lemma}

\begin{proof} To simplify the notation we will omit the index $t$. Let $\delta(s)$ be a generic orbit of the vector field $S$ (here by a generic orbit we mean an orbit that is not contained in any coordinate plane). We can parametrize $\delta(s)$ as $s\to(as^{\beta\gamma},bs^{\alpha\gamma},cs^{\alpha\beta})$, $a\neq0,b\neq0,c\neq0$. Without loss of generality we can suppose that $a=b=c=1$. We have $$f_{t}(\delta(s))=[(s^{\beta\gamma}u_{0}+s^{\alpha(\beta+\gamma)}h_{0})^\alpha:(s^{\alpha\gamma}u_{1}+s^{\beta(\alpha+\gamma)}h_{1})^{\beta}:(s^{\alpha\beta}u_{2}+s^{\gamma(\alpha+\beta)}h_{2})^\gamma].$$ Condition $\beta\gamma\leq\alpha(\beta+\gamma)$ implies that we can extract the factor $s^{\alpha\beta\gamma}$ from $f_{t}(\delta(s))$.
 
Hence we obtain \begin{equation}\label{eq4}f_{t}(\delta(s))=[(u_{0}+s^{k}h_{0})^\alpha:(u_{1}+s^{l}h_{1})^{\beta}:(u_{2}+s^{m}h_{2})^\gamma]\end{equation}
where $k=\alpha(\beta+\gamma)-\beta\gamma$, $l=\beta(\gamma+\alpha)-\alpha\gamma$ and $m=\gamma(\alpha+\beta)-\alpha\beta$.

Since $V_{\tau}$ is a fiber, $f_{0}(V_{\tau})=[d:e:f]\in \mathbb P^2$ with $d\neq0,e\neq0,f\neq0$. If we take a covering of $I(f)=\{p_{1},..., p_{\frac{\nu^3}{\alpha\beta\gamma}}\}$ by small open balls $B_{j}(p_{j})$, $1\leq j\leq\frac{\nu^3}{\alpha\beta\gamma}$, the set $V_{\tau}\backslash\cup_{j}B_{j}(p_{j})$ is compact. For a small deformation $f_{t}$ of $f_{0}$ we have that $f_{t}[V_{\tau}(t)\backslash\cup_{j}B_{j}(p_{j})(t)]$ stays near $f[V_{\tau}\backslash\cup_{j}B_{j}(p_{j})]$. Hence for $t$ sufficiently small the components of expression \ref{eq4} do not vanish both inside as well as outside of the neighborhood $\cup_{j}B_{j}(p_{j})(t)$. 
 
This implies that the components of $f_{t}$ do not vanish along each generic fiber that extends locally as a singular curve of the foliation $\mathcal F_{t}$. This is possible only if $f_{t}$ is constant along these curves. In fact, $f_{t}(V_{\tau}(t))$ is either a curve or a point. If it is a curve then it cuts all lines of $\mathbb P^2$ and therefore the components should be zero somewhere. Hence $f_{t}(V_{\tau}(t))$ is constant and we conclude that $V_{\tau}(t)$ is a fiber.  
 
Observe also that when we make a blow-up with weights $(\beta\gamma,\alpha\gamma,\alpha\beta)$ at the points of $I(f_{t})$ we solve completely the indeterminacy points of the mappings $f_{t}$ in the case  $\beta\gamma\leq\alpha(\beta+\gamma)$ for each $t$.
\end{proof}

When $\beta\gamma>\alpha(\beta+\gamma)$ the situation requires more detail. Suppose that the orbits contained in the coordinate planes that extends globally as singular curves of the foliation $\mathcal F_{t}$ are fibers of $f_{t}$. To simplify assume also that g.c.d$(\alpha,\beta)=1$, g.c.d$(\alpha,\gamma)=1$ and g.c.d$(\gamma,\beta)=1$ (the general case is similar). We can assume without loss of generality that this orbit is contained in the coordinate plane $x_{0}(t)=0$. In this case  the orbit is of the form $(x_{0}=x_{1}^\beta-cx_{2}^\gamma=0)$. We have that the germ of $f_{0,t}$ at the point $p_{j}(t)$ belongs to the ideal generated by $x_{0}(t)$ and $(x_{1}^\beta-cx_{2}^\gamma)(t)$. Hence we can write the function $h_{0t}$ of expression \label{eq3} as
$$h_{0t}=x_{0}(t)h_{01t}+(x_{1}^\beta(t)-cx_{2}^\gamma(t))h_{02t}$$
where $h_{01t}, h_{02t} \in \mathcal O_{2}$.  Therefore we can repeat the argument of Lemma \ref{casomenor} and extract the factor $s^{\alpha\beta\gamma}$. We conclude, as above, that $V_{\tau}(t)$ is also a fiber when we have $\beta\gamma>\alpha(\beta+\gamma)$.

Hence to complete the proof for the case $\beta\gamma>\alpha(\beta+\gamma)$ we need the following result:
\begin{lemma}  \label{casomaior}  If $\beta\gamma>\alpha(\beta+\gamma)$ then any orbit of the vector field $S$ contained in some coordinate plane at $p_{j}(t)$ and which extends globally as a singular curve of the foliations $\mathcal F_{t}$ is a fiber of the mapping $f_{t}$ for fixed $t$. 
\end{lemma}

\begin{proof} Denote $({f}_{t})_{t \in D_{\epsilon'}}:\mathbb P^{3} \DashedArrow[->,densely dashed    ]\mathbb P^{2}$ by $f_{t}=[P_{0}^{\alpha}(t):P_{1}^{\beta}(t):P_{2}^{\gamma}(t)] $. As previously, let us consider an orbit of the vector field $S$ on a small neighborhood of an indeterminacy point of $f_{t}$, $B_{j}(p_{j}(t))$, $1\leq j\leq\frac{\nu^3}{\alpha\beta\gamma}$ and denote by $V_{\tau}(t)$ the global extension of this orbit to $\mathbb P^3$. Without loss of generality we can assume that the orbit is contained in the plane $(x_{0}(t)=0)$ and we can suppose that it can be parametrized as $s\to(0,s^{\gamma},s^{\beta})$. To simplify the notation we will omit the index $t$ in some expressions. After evaluating the mapping $f_{t}$ on this orbit, on a neighborhood of $p_{j}(t)$ we obtain:
$$f_{t}(\delta(s))=[s^{\alpha(\beta+\gamma)}h_{0}^\alpha:s^{\beta\gamma}u_1^{\beta}:s^{\beta\gamma}u_2^{\gamma}].$$ 
This can be written as 
\begin{equation}\label{eq5}
[s^{\alpha(\beta+\gamma)}\tilde{h_{0}}:s^{\beta\gamma}u_1^{\beta}:s^{\beta\gamma}u_2^{\gamma}]=[X(s):Y(s):Z(s)].
\end{equation} 

Firstly we prove that $f_{t}(V_{\tau}(t))$ is contained in a line of the form $(Y-\lambda Z=0)$ of $\mathbb P^2.$
Let us consider the meromorphic function with values in $\mathbb{P}^1$ given by $g_{t}(s)=\frac{Z(s)}{Y(s)}=\frac{u_2^{\gamma}}{u_1^{\beta}}$. When $s\to0$ this function goes to a constant $\lambda\neq0,\lambda\neq\infty$.
Observe that  for small $t$ the function $\frac{P_{1}^{\beta}}{P_{2}^{\gamma}}(t): V_{\tau}(t)\backslash\cup_{j}B_{j}(p_{j}(t)) \rightarrow \mathbb{P}^1$ stays near $\frac{P_{1}^{\beta}}{P_{2}^{\gamma}}(0): V_{\tau}(0)\backslash\cup_{j}B_{j}(p_{j}(0)) \rightarrow \mathbb{P}^1$. Note that since $V_{\tau}(0)$ is a fiber $\frac{P_{1}^{\beta}}{P_{2}^{\gamma}}(0)$ does not vanish. We conclude that $f_{t}(V_{\tau}(t))\subset(Y-\lambda Z=0)\simeq\mathbb P^1$.

If $\beta\gamma>\alpha(\beta+\gamma)$ we can write Equation \ref{eq5} as 
$$[\tilde{h}_{0}(s):s^m u_1^{\beta}:s^m u_2^{\gamma}]$$
where $m=\beta\gamma-\alpha(\beta+\gamma)$. Observe that when $s=0$ the function  $\tilde{h}_{0}(s)$ could vanish; in this case such a point corresponds to a indeterminacy point $p_{j}(t)$ of $f_{t}$ for some $j$.  At $p_{j}(t)$ we can write the first component of Equation \ref{eq5} as $\tilde{h}_{0}(s)=s^{\rho_j}\tilde{h}_{j}(s)$ where either $\tilde{h}_{j}(s) \in \mathcal O^{\ast}(\mathbb C,0)$ or $\tilde{h}_{0}\equiv0$; however in the second case we are done, that is, $V_{\tau}(t)$ is a fiber of $f_{t}$. At $p_{j}(t)$ we have two possibilities:

First case: ${\rho_j}< m$. In this case we can write Equation \ref{eq5} as
 \begin{equation}\label{eq6}
[\tilde{h}_{j}(s):s^{m-{\rho_j}} u_1^{\beta}:s^{m-{\rho_j}} u_2^{\gamma}].
\end{equation}
If $s\to0$ the image goes to $[1:0:0]$, hence ${f_{t}}|_{V_{\tau}(t)}(p_{j}(t))=[1:0:0]$.

Second case: ${\rho_j}\geq m$. We can write Equation \ref{eq6} as
\begin{equation}\label{eq7}
[s^{{\rho_j}-m}\tilde{h}_{j}(s):  u_1^{\beta}: u_2^{\gamma}].
\end{equation} If $s\to0$ the image goes to $[a:\lambda:1]$ where $a\in\mathbb C$. This is due to the fact that the image of such a point belongs to the curve $(Y-\lambda Z=0)\simeq\mathbb P^1$ and hence we can write it as $[a:\lambda:1]$. Suppose that $f_{t}|_{V_{\tau}(t)}$ is not constant and consider the mapping $f_{t}|_{V_{\tau}(t)}:V_{\tau}(t) \to f_{t}(V_{\tau}(t))\subset(Y-\lambda Z=0)$ for fixed $t$. Denote $Q=\{j|\rho_{j}<m\}$. Note that $p\in V_{\tau}(t)$ and $f_{t}|_{V_{\tau}(t)}(p)=[1:0:0]$ imply that $p=p_j(t)$ for some $j \in Q$; that is, $(f_{t}|_{V_{\tau}(t)})^{-1}[1:0:0]=\{p_{j}(t),j \in Q\}$. Moreover, by Equation \ref{eq7} we have $mult(f_{t}|_{V_{\tau}(t)},p_j(t))=m-\rho_j$. In particular, the degree of $f_{t}|_{V_{\tau}(t)}$ is $$deg(f_{t}|_{V_{\tau}(t)})=\sum_j(m-\rho_j).$$ On the other hand, if $p\in(f_{t}|_{V_{\tau}(t)})^{-1}[0:\lambda:1]$ then $(P_{0}^{\alpha}(p)=0)$ and so $mult(f_{t}|_{V_{\tau}(t)},p)$ is equal to the intersection number of $(P_{0}^{\alpha}(t)=0)$ and $V_\tau(t)$ at $p$. Hence 
$$deg(f_{t}|_{V_{\tau}(t)})={V_{\tau}(t)}.P_{0}^{\alpha}(t)=deg({V_{\tau}(t)})\times deg(P_{0}^{\alpha}(t))=\frac{\nu^3}{\alpha}=\sum_j(m-\rho_j).$$ 
But $(m-\rho_j)\leq m=\beta\gamma-\alpha(\beta+\gamma)$ and so 
  $$\sum_{j\in Q}(m-\rho_j)\leq\# Q\times m\leq\frac{\nu^3}{\alpha\beta\gamma}\times (\beta\gamma-\alpha(\beta+\gamma))={\nu^3}(\frac{1}{\alpha}-\frac{1}{\beta}-\frac{1}{\gamma})$$
  which implies that $\frac{1}{\alpha}\leq\frac{1}{\alpha}-\frac{1}{\beta}-\frac{1}{\gamma}$ and we arrive to a contradiction. Therefore, $Q=\emptyset$, $f_{t}|_{V_{\tau}(t)}$ is a constant and ${V_{\tau}(t)}$ is a fiber of $f_{t}$. \end{proof}

\subsubsection{Part 2} Let us now define a family of foliations $(\mathcal {G}_{t})_{t\in D_{\epsilon}}, \mathcal {G}_{t} \in  \mathcal A$ (see Section \ref{section4}) such that  $\mathcal{F}_{t}= f_{t}^{ \ast}(\mathcal{G}_{t})$ for all $t\in D_{\epsilon}$.  Firstly we consider the case $n=3.$ Let $M_{[\beta \gamma,\alpha \gamma,\alpha \beta]}(t)$ be the family of ``complex algebraic threefolds" obtained from $\mathbb P^3$ by blowing-up with weights $(\beta \gamma,\alpha \gamma,\alpha \beta)$ at the $\frac{\nu^{3}}{\alpha \beta \gamma}$ points $p_{1}(t),..., p_{j}(t),...,p_{\frac{\nu^{3}}{\alpha \beta \gamma}}(t)$ corresponding to $I(t)$ of $\mathcal F_{t}$; and denote by $$\pi_{w}(t):M_{[\beta \gamma,\alpha \gamma,\alpha \beta]}(t) \to \mathbb P^3$$ the blowing-up map. The exceptional divisor of $\pi_{w}(t)$ consists of ${\frac{\nu^{3}}{\alpha \beta \gamma}}$ orbifolds \break $E_{j}(t)=\pi_{w}(t)^{-1}(p_{j}(t)),$ $1\leq j \leq {\frac{\nu^{3}}{\alpha \beta \gamma}}$, which are weighted projective planes of the type $\mathbb P_{[\beta \gamma,\alpha \gamma,\alpha \beta]}^2$. Each of these has three lines of singular points of $M_{[\beta \gamma,\alpha \gamma,\alpha \beta]}(t)$, all isomorphic to weighted projective lines, but these singularities will not interfere our arguments (for more detail see \cite{mamor} ex. 3.6 p 957).

More precisely, if we blow-up $\mathcal{F}_{t}$ at the point $p_{j}(t)$, then the restriction of the strict transform $\pi_{w}^{\ast}\mathcal{F}_{t}$ to the exceptional divisor $E_{j}(t)=\mathbb P_{[\beta \gamma,\alpha \gamma,\alpha \beta]}^2$ is the same quasi-homogeneous $1$-form that defines $\mathcal{F}_{t}$ at the point $p_{j}(t)$. We have that $E_{j}(t)$ is birationally equivalent to $\mathbb P^2$; by an abuse of notation, we will denote this property by $E_j(t) \simeq {\mathbb P}^2$. It follows that we can push-forward the foliation to $\mathbb P^2$. With this process we produce a family of holomorphic foliations in $\mathcal A$. This family is the ``holomorphic path'' of candidates to be a deformation of $\mathcal G_{0}$. In fact, since $\mathcal A$ is an open set we can suppose that this family is inside $\mathcal A$. 
We fix the exceptional divisor $E_1(t)$ to work with and we denote by $\mathcal{G}_t$ the restriction of $\pi_w^*\mathcal{F}_t$ to $E_1(t)$. As we have seen, this process produces foliations in $\mathcal A$ up to a linear automorphism of $\mathbb P^2$. Consider the family of mappings ${f}_{t}:\mathbb P^{3} \DashedArrow[->,densely dashed    ]\mathbb P^{2}$, ${t \in D_{\epsilon'}}$ defined in Proposition \ref{recupmapas}. We will consider the family $(f_{t})_{t\in D_{\epsilon}}$ as a family of rational maps  $f_{t}:\mathbb P^3  \DashedArrow[->,densely dashed    ] E_{1}(t)$; we decrease $\epsilon$ if necessary. Note that the map $${f}_{t}\circ \pi_{w}(t):M_{[\beta \gamma,\alpha \gamma,\alpha \beta]}(t) \backslash \cup_{j} E_{j}(t) \to E_{1}(t) \simeq\mathbb P^2 $$ extends holomorphically, that is, as an orbifold mapping, to 
$$\hat{f}_{t}:M_{[\beta \gamma,\alpha \gamma,\alpha \beta]}(t) \to  \mathbb P_{[\beta \gamma,\alpha \gamma,\alpha \beta]}^2 \simeq E_{1}(t)\simeq\mathbb P^2.$$
This is due to the fact that each orbit of the vector field $S_t$ determines an equivalence class in $\mathbb P_{[\beta \gamma,\alpha \gamma,\alpha \beta]}^2$ and is a fiber of the map $$(x_{0}(t),x_{1}(t),x_{2}(t))\to(x_{0}^\alpha(t), x_{1}^\beta(t),x_{2}^\gamma(t)).$$

The mapping  $f_t$ can be interpreted as follows. Each fiber of $f_t$ meets $p_{j}(t)$ once, which implies that each fiber of $\hat{f}_{t}$ cuts $E_{1}(t)$ once outside of the three singular curves in $[M_{[\beta \gamma,\alpha \gamma,\alpha \beta]}(t) \cap E_{1}(t)]$. Since $M_{[\beta \gamma,\alpha \gamma,\alpha \beta]}(t) \backslash \cup_{j} E_{j}(t)$ is biholomorphic to $\mathbb P^3 \backslash I(t)$, after identifying $E_{1}(t)$ with $\mathbb P_{[\beta \gamma,\alpha \gamma,\alpha \beta]}^2$, we can imagine that if $q \in M_{[\beta \gamma,\alpha \gamma,\alpha \beta]}(t) \backslash \cup_{j} E_{j}(t)$ then $\hat{f}_{t}(q)$ is the intersection point of the fiber $\hat{f}_{t}^{-1}(\hat{f}_{t}(q))$ with $E_{1}(t)$. We obtain a mapping $$\hat{f}_{t}:M_{[\beta \gamma,\alpha \gamma,\alpha \beta]}(t) \to \mathbb P^2.$$
It can be extended over the singular set of $M_{[\beta \gamma,\alpha \gamma,\alpha \beta]}(t)$ using Riemann's  Extension Theorem. This is due to the fact that the orbifold $M_{[\beta \gamma,\alpha \gamma,\alpha \beta]}(t)$ has singular set of codimension $2$ and these singularities are of the quotient type; therefore it is a normal complex space. We shall also denote this extension by $\hat{f}_{t}$ to simplify the notation. Observe that the blowing-up with weights $(\beta \gamma,\alpha \gamma,\alpha \beta)$ can completely solve the indeterminacy set of $f_t$ for each $t$. 

With all these ingredients we can define the foliation $\tilde{\mathcal{F}}_{t}={f}_{t}^\ast(\mathcal G_{t}) \in PB(\Gamma-1,\nu,\alpha,\beta,\gamma)$. This foliation is a deformation of $\mathcal{F}_0$.  Based on the previous discussion let us denote ${\mathcal{F}_{1}}({t})= \pi_{w}(t)^{\ast} ({\mathcal{F}}_{t})$ and $\hat{\mathcal{F}_{1}}({t})= \pi_{w}(t)^{\ast} (\tilde{\mathcal{F}}_{t})$. 
\begin{lemma}If ${\mathcal{F}_{1}}({t})$ and $\hat{\mathcal{F}_{1}}({t})$ are the foliations defined previously, we have that $${\mathcal{F}_{1}}({t})|_{{E_{1}(t)}\simeq \mathbb P_{[\beta \gamma,\alpha \gamma,\alpha \beta]}^2}={\hat {\mathcal G}_{t}}=\hat{\mathcal{F}_{1}}({t})|_{{E_{1}(t)}\simeq \mathbb P_{[\beta \gamma,\alpha \gamma,\alpha \beta]}^2}$$ where ${\hat {\mathcal G}_{t}}$ is the foliation induced on ${E_{1}(t)}\simeq \mathbb P_{[\beta \gamma,\alpha \gamma,\alpha \beta]}^2$ by the quasi-ho\-mo\-ge\-ne\-ous $1$-form $\eta_{p_{1}(t)}$. 
\end{lemma}

\begin{proof} In a neighborhood of $p_1(t) \in I(t)$, ${\mathcal{F}}_{t}$ is represented by the quasi-homogeneous $1$-form $\eta_{p_{1}(t)}$. This $1$-form satisfies $i_{S_t}\eta_{p_{1}(t)}=0$ and therefore naturally defines a foliation on the weighted projective space $E_{1}(t)\simeq \mathbb P_{[\beta \gamma,\alpha \gamma,\alpha \beta]}^2$. This proves the first equality. The second equality follows from the geometrical interpretation of the mapping $\hat{f}_{t}:M_{[\beta \gamma,\alpha \gamma,\alpha \beta]}(t) \to  \mathbb P_{[\beta \gamma,\alpha \gamma,\alpha \beta]}^2 \simeq \mathbb P^2$, since $\hat{\mathcal{F}_{1}}({t})=f_{1}(t)^{*}(\mathcal {G}_{t})$. \end{proof}

Now we use the fact that $\mathbb P_{[\beta \gamma,\alpha \gamma,\alpha \beta]}^2 \simeq \mathbb P^2$  to obtain the equality 
$${\mathcal{G}_{t}}={\mathcal{F}_{1}}({t})|_{{E_{1}(t)}\simeq \mathbb P^2} =\hat{\mathcal{F}_{1}}({t})|_{{E_{1}(t)}\simeq \mathbb P^2}.$$ Let ${\tau}_{1}(t)$ be a singularity of $\mathcal G_{t}$ outside the three invariant straight lines. Since the map $t \to {\tau}_{1}(t) \in \mathbb P^2$ is holomorphic, there exists a holomorphic family of automorphisms of $\mathbb P^2$, $t \to H(t)$ such that  ${\tau}_{1}(t)=[a:b:c]$ $\in {{E_{1}(t)}\simeq \mathbb P^2}$  is kept fixed. Observe that such a singularity has non algebraic separatrices at this point. Fix a local analytic coordinate system $(x_t,y_t)$ at 
${\tau}_{1}(t)$ such that the local separatrices are $(x_t=0)$ and $(y_t=0)$, respectively. Observe that the local smooth hypersurfaces along $\hat V_{\tau_1(t)}=\hat{f}_{t}^{-1}({\tau}_{1}(t))$ defined by $\hat X_t:=(x_t\circ\hat{f}_{t}=0)$ and 
$\hat Y_t:=(y_t\circ\hat{f}_{t}=0)$ are invariant for $\hat{\mathcal{F}_{1}}({t})$. Furthermore, they meet transversely along  $\hat V_{\tau_1(t)}$. On the other hand, $\hat V_{\tau_1(t)}$ is also contained in the Kupka set of ${\mathcal{F}_{1}}({t})$.  Therefore there are two local smooth hypersurfaces $X_t:=(x_t\circ\hat{f}_{t}=0)$ and 
$Y_t:=(y_t\circ\hat{f}_{t}=0)$ invariant for ${\mathcal{F}_{1}}({t})$ such that:
\begin{enumerate}
\item $X_t$ and $Y_t$ meet transversely along $\hat V_{\tau_1(t)}$.
\item$X_t\cap\pi_{w}(t)^{-1}(p_{1}(t))=(x_t=0)=\hat X_t\cap\pi_{w}(t)^{-1}(p_{1}(t))$ and \break $Y_t\cap\pi_{w}(t)^{-1}(p_{1}(t))=(y_t=0)=\hat Y_t\cap\pi_{w}(t)^{-1}(p_{1}(t))$ (because ${\mathcal{F}_{1}}({t})$ and $\hat{\mathcal{F}_{1}}({t})$) coincide on ${{E_{1}(t)}\simeq \mathbb P^2}$).
\item $X_t$ and $Y_t$ are deformations of $X_0=\hat X_0$ and $Y_0=\hat Y_0$, respectively. 
\end{enumerate}

\begin{lemma}\label{lemafund}$X_t=\hat X_t$ for small $t$.
\end{lemma}
\begin{proof} Let us consider the projection $\hat{f}_{t}:M_{[\beta \gamma,\alpha \gamma,\alpha \beta]}(t) \to  \mathbb P_{[\beta \gamma,\alpha \gamma,\alpha \beta]}^2 \simeq \mathbb P^2$ on a neighborhood of the regular fibre  $\hat V_{\tau_1(t)}$, and fix local coordinates $x_t,y_t$ on $\mathbb P^2$ such that $X_t:=(x_t\circ\hat{f}_{t}=0)$. For small $\epsilon$, let $H_{\epsilon}=(y_t\circ\hat{f}_{t}=\epsilon)$. Thus $\hat\Sigma_\epsilon=\hat X_t\cap H_\epsilon$ are (vertical) compact curves, deformations of $\hat\Sigma_0=\hat V_{\tau_1(t)}$. Set $\Sigma_\epsilon=X_t\cap\hat H_\epsilon$. The $\Sigma_\epsilon's$, as the $\hat\Sigma_\epsilon's$, are compact curves (for $t$ and $\epsilon$ small), since $X_t$ and $\hat X_t$ are both deformations of the same $X_0$. Thus for small $t$, $X_t$ is close to $\hat X_t$. It follows that $\hat{f}_{t}(\Sigma_\epsilon)$ is an analytic curve contained in a small neighborhood of ${\tau_1(t)}$, for small $\epsilon$. By the maximum principle, we must have that $\hat{f}_{t}(\Sigma_\epsilon)$ is a point, so that  $\hat{f}_{t}(X_t)=\hat{f}_{t}(\cup_{\epsilon}\Sigma_\epsilon)$ is a curve $C$, that is, $X_t=\hat{f}_{t}^{-1}(C)$. But $X_t$ and $\hat X_t$ intersect the exceptional divisor ${{E_{1}(t)}\simeq \mathbb P^2}$ along the separatrix $(x_t=0)$ of $\mathcal G_t$ through ${\tau_1(t)}$. This implies that $X_t=\hat{f}_{t}^{-1}(C)=\hat{f}_{t}^{-1}(x_t=0)=\hat X_t$.
\end{proof}
We have proved that the foliations $\mathcal F_t$ and $\tilde{\mathcal F_t}$ have a common local leaf: the leaf that contains $\pi_{w}(t)\left(X_t\backslash\hat V_{\tau_1(t)}\right)$ which is not algebraic. Let $D(t):=Tang(\mathcal{F}(t),\hat {\mathcal{F}}(t))$ be the set of tangencies between $\mathcal{F}(t)$ and $\hat {\mathcal{F}}(t)$. This set can be defined by $D(t)=\{ Z \in \mathbb {C}^4; \Omega(t) \wedge\hat {\Omega}(t)=0\},$ where  $\Omega(t)$ and  $\hat{\Omega}(t)$ define $\mathcal{F}(t)$ and $\hat{\mathcal{F}}(t)$, respectively. Hence it is an algebraic set. Since this set contains an immersed non-algebraic surface $X_t$, we necessarily have that  $D(t)=\mathbb P^3.$ This proves Theorem B in the case $n=3.$

Suppose now that $n\geq4$. The previous argument implies that if $\Upsilon$ is a generic $3-$plane in $\mathbb P^n$, we have $\mathcal{F}(t)_{|\Upsilon}=\hat{\mathcal{F}}(t)_{|\Upsilon}$. In fact, such planes cut transversely every strata of the singular set, and $I(t)$ consists of $\frac{\nu^{3}}{\alpha \beta \gamma} $ points. This implies that $f_{t}$ is generic for $|t|$ sufficiently small. We can then repeat the previous argument, finishing the proof of Theorem B.

Recall from Definition \ref{generic} the concept of a generic map. Let \break $f  \in BRM\left(n,\nu,\alpha,\beta,\gamma\right)$, $I(f)$ its indeterminacy locus and $\mathcal F$ a foliation on $\mathbb P^n$, $n\geq3$. Consider the following properties:

\begin{center}
\begin{minipage}{10cm}
$\mathcal{P}_1:$ If n=3, at any point $p_{j}\in I(f)$ $\mathcal F$ has the following local structure: there exists an analytic coordinate system $(U^{p_{j}},Z^{p_{j}})$ around $p_{j}$ such that $Z^{p_{j}}(p_{j})=0 \in ({\mathbb C^3,0})$ and $\mathcal F|_{(U^{p_{j}},Z^{p_{j}})}$ can be represented by a quasi-homogenous $1$-form $\eta_{p_{j}}$ (as described in the Lemma \ref{lema5.6}) such that
\begin{enumerate} 
\item[(a)] $Sing(d \eta_{p_{j}}) = {0}$,
\item[(b)] $0$ is a quasi-homogeneous singularity of the type $\left [{\beta}.{\gamma}:{\alpha}.{\gamma}:{\alpha}.{\beta}  \right ]$.
\end{enumerate}
If $n\geq4$, $\mathcal F$ has a local structure product: the situation for n=3 ``times'' a regular foliation in ${\mathbb C^{n-3}}$.
\end{minipage}
\end{center}

\begin{center}
\begin{minipage}{10cm}
$\mathcal{P}_2:$ There exists a fibre $f^{-1}(q)=V(q)$ such that $V(q)=f^{-1}(q)\backslash I(f)$ is contained in the Kupka-Set of $\mathcal F$ and $V(q)$ is not contained in $\bigcup_{i=0} ^ {i=2} (F_{i}=0)$.
\end{minipage}
\end{center}

\begin{center}
\begin{minipage}{10cm}
$\mathcal{P}_3:$ $V(q)$ has transversal type $X$, where $X$ is a germ of vector field on $({\mathbb C^2,0})$ with a non algebraic separatrix and such that $ 0 \in{\mathbb C^2}$ is a non-degenerate singularity with eigenvalues $\lambda_{1}$ and $\lambda_{2}$, $\frac{\lambda_{2}}{\lambda_{1}} \notin {\mathbb R}$.
\end{minipage}
\end{center}

Lemma \ref{lemafund} allows us to prove the following result:

\begin{main}\label{teoc}
In the conditions above, if properties $\mathcal{P}_1$, $\mathcal{P}_2$ and $\mathcal{P}_3$ hold then $\mathcal F$ is a pull back foliation, $\mathcal {F}= f^{*}(\mathcal{G})$, where $\mathcal{G}$ is of degree $d\geq2$ on $\mathbb P^2$ with three invariant lines in general position.  
\end{main}
\begin{remark} Note that when $\mathcal G$ does not have invariant algebraic curves and we perform the pull-back by $f$ as above, the indeterminacy set of $f$ does not satisfy the hypothesis of the theorem. Take for example $\mathcal G$ being Jouanoulou's foliation. In the case $n=3$ the indeterminacy set of $f$ is not a quasi-homogeneous singularity of  $\mathcal {F}= f^{*}(\mathcal{G}).$
 \end{remark}

\section{Appendix}

\subsection{The pull back's foliation degree}

We refer the reader to \cite{bea} for the basic theory of algebraic surfaces. We recall that $NS_{\mathbb Z}(\mathbb P^2)$ is the N\'eron-Severi group of divisors in $\mathbb P^2$. Denote by  $\mathcal {D}_{F}$ the divisor determined by the vanishing of the Jacobian determinant of  $F: {\mathbb P^2} \to {\mathbb P^2}$. It is locally defined by the vanishing of det $DF$, where $DF$ denotes de differential of $F$. Its support is the critical set of $F,$ and it satisfies the equation $$ K_{\mathbb P^2}=F^{*}K_{\mathbb P^2}+\mathcal {D}_{F}$$ where $K_{\mathbb{P}^2}$ is the canonical divisor. Note that $\mathcal {D}_{F}$ is a divisor in the domain of $F.$ It is a well known fact that a singular holomorphic foliation $\mathcal G$ on $\mathbb P^2$ has a cotangent bundle given by $T^*_\mathcal{G}=\mathcal O_{\mathbb P^2}(d-1)$ where $d$ is the foliation's degree \cite{br}.

The next result is part of a proposition extracted from \cite{Fa-Pe} pp 5. The statement used here is sufficient for our purposes.
 
\begin{proposition} \label{favrejorge}Suppose $F:(\mathbb P^2,\mathcal H)\to(\mathbb P^2,\mathcal G)$ is a dominant rational map. 
 Then one has $$F^*T^*_\mathcal{G} = T^*_\mathcal{H} - \mathcal {D}$$ in $NS_{\mathbb Z}(\mathbb P^2)$ for some (on necessarily effective) divisor with support included in the critical set of $F$ and satisfying $\mathcal {D} \leq \mathcal {D}_{F}.$  
\end{proposition}

The idea of the proof of Proposition \ref{graupargenerico} is the following: since $I(f)$ is a set of codimension $3$, by Bertini's theorem we can embed a generic $\mathbb P^2$  on $\mathbb P^n$ in such a way that it doesn't intersect $I(f).$ To be more clear in our discussion we will denote the embeded $\mathbb{P}^2$ by $\Delta$. We consider the restriction  $\mathcal{F}|_{\Delta}=\mathcal H.$ Then we concentrate on the study of $F=f |_{\Delta}: (\Delta, \mathcal H) \mapsto (\mathbb P^2 , \mathcal{G})$.

\begin{proof} Note that $F$ is in $BRM(2,\nu,\alpha,\beta,\gamma),$ since $\Delta$ is given by homogeneous linear equations on $\mathbb C^{n+1}.$
Since $\Delta \cap I(f)= \emptyset$, $F$ is holomorphic. Denote the zero divisor of $F_i$ by $\tilde D_{i}$ and the zero divisor associated to $(dF_{0}\wedge dF_{1}\wedge dF_{2})$ by $\mathcal {D}'$. From Proposition \ref{favrejorge} we have that the divisor $D$ consists of $\mathcal D_{F}-D_{0}$, where $\mathcal D_{F}=(\alpha -1)\tilde D_{0} + (\beta-1)\tilde D_{1}+(\gamma -1)\tilde D_{2} + \mathcal {D}'$ . Recall that $$DF=(\alpha\beta\gamma)[(F_{0}^{\alpha -1}F_{1}^{\beta-1}F_{2}^{\gamma-1})(dF_{0}\wedge dF_{1}\wedge dF_{2})].$$

On the other hand since $\mathcal G$ has only $3$ invariant algebraic curves and of its singularities are of Poincar\'e-type then the pull-back of the divisor $D_{0}$ is equal to $$(\alpha -1)\tilde D_{0} + (\beta-1)\tilde D_{1}+(\gamma -1)\tilde D_{2}.$$ This implies that  $D=\mathcal D '.$ Hence $T^{*} _{\mathcal {H}} = F^{*}T^{*} _{\mathcal {G}}+ \mathcal {D}'$.

Now denote by $d'$ the degree of $\mathcal {H}$. We have that $$d'-1=\nu(d-1)+\nu(\frac{1}{\alpha}+\frac{1}{\beta}+\frac{1}{\gamma})-3$$ and the result follows. 
\end{proof}

\bibliographystyle{amsalpha}

\end{document}